\documentclass[10pt]{article}
\usepackage{authblk}
\usepackage{amssymb, amscd, amsmath, amsthm, amsfonts,color}
\allowdisplaybreaks[2]
\usepackage[margin=2cm]{geometry}

\theoremstyle{definition}

\numberwithin{equation}{section}

\newtheorem{thm}{Theorem}[section]
\newtheorem*{thm*}{Theorem}

\newtheorem*{defn*}{Definition}

\newtheorem*{propn*}{Proposition}
\newtheorem{lemma}[thm]{Lemma}
\newtheorem*{lemma*}{Lemma}
\newtheorem{cor}[thm]{Corollary}
\newtheorem*{cor*}{Corollary}

\newtheorem*{eg*}{Example}

\newcommand{\bbC}{\mathbb{C}}

\newcommand{\bbZ}{\mathbb{Z}}

\newcommand{\summn}{\sum_{m_2\in\mathbb{Z}}}
\newcommand{\summpnp}{\sum_{n_2\in\mathbb{Z}}}
\newcommand{\summppnpp}{\sum_{n_3\in\mathbb{Z}}}
\newcommand{\summnmpnp}{\sum_{m_2,n_2 \in \mathbb{Z}}}
\newcommand{\summpnpmppnpp}{\sum_{ n_3,n_2 \in \mathbb{Z}}}
\newcommand{\summnmpnpmppnpp}{\sum_{m_2,n_2,n_3\in \mathbb{Z}}}

\def\p#1{\frac{\partial}{\partial #1}} 

\begin{document}
\title{A family of representations of the affine Lie superalgebra $\widehat{\mathfrak{gl}_{m|n}}(\bbC)$}
\author[1]{Yongjie Wang}
\author[1]{Hongjia Chen\footnote{ Correspondence: hjchen@ustc.edu.cn.}
\setcounter{footnote}{-1}\footnote{The second author is partially supported by the Recruitment Program of Global Youth Experts of China,
by the start-up funding from University of Science and Technology of China, by NSF of China (Grant 11401551). The third author is partially supported by NSERC of Canada.}}
\author[2]{Yun Gao}

\affil[1]{School of Mathematical Sciences, University of Science and Technology of China, Hefei, China 230026}
\affil[2] {Department of Mathematics and Statistics, York University, Toronto, M3J 1P3, Canada}

\maketitle

\begin{abstract}
In this paper, we used the free fields of Wakimoto to construct a class of irreducible representations
for the general linear Lie superalgebra $\mathfrak{gl}_{m|n}(\bbC)$.
The structures of the representations over the general linear Lie superalgebra and
the special linear Lie superalgebra are studied in this paper.
Then we extend the construction to the affine Kac-Moody Lie superalgebra
$\widehat{\mathfrak{gl}_{m|n}}(\bbC)$ on the tensor product of
a polynomial algebra and an exterior algebra with infinitely many variables
involving one parameter $\mu$, and we also obtain the necessary and sufficient
condition for the representations to be irreducible.
In fact, the representation is irreducible if and only if the parameter $\mu$ is nonzero.

\vspace{0.2cm} \noindent{\textbf{Keywords}}: affine Lie superalgebra; irreducible representation; free fields of Wakimoto; polynomial algebra; exterior algebra.

\vspace{0.2cm} \noindent{\textbf{Mathematics Subject Classification 2010}}: 17B10, 17B67.
\end{abstract}

\section{Introduction}
The theory of Lie superalgebras and their representations plays a fundamental role
in the understanding and exploitation of supersymmetry in physical systems.
The classification of the simple complex finite-dimensional Lie superalgebras was completed by Kac in \cite{K}.
Since then, these superalgebras, particularly of type $A(m,n)$,
have found applications in various areas including quantum mechanics, nuclear physics, particle physics, and string theory.
Understanding their module theory, however, has been a very difficult problem,
even at the level of the finite-dimensional simple modules for the Lie superalgebras of type $A(m,n)$.
As for the structure and representations of infinite-dimensional Lie superalgebras, these are understood only in particular cases.

People have constructed a great number of various irreducible representations for affine Lie superalgebras from different contexts.
It is certainly not possible to classify all irreducible modules for affine Lie superalgebras
but people are able to classify irreducible modules which have certain  good and natural properties.
For example, Rao and Zhao and others (see \cite{RZ} and references therein) classified
all irreducible integrable modules with finite-dimensional weight spaces for the  affine superalgebras
which are not type $A(m,n)$ or type $C(m),$
who proved that such modules comprise of irreducible integrable highest weight modules,
irreducible integrable lowest weight modules and evaluation modules.
Recently, Wu and Zhang (see \cite{WZ}) classified all irreducible integrable modules
for affine Lie superalgebras of type $A(m,n)$ and type $C(m)$, more precisely,
there is a new class of integrable modules for affine Lie superalgebras of type $A(m,n)$ and type $C(m)$
which must be of highest weight type, but are not necessarily evaluation modules.
Serganova in \cite{S} studied the affine Kac-Moody Lie superalgebras $\mathfrak{sl}_{1|n}^{(1)}$ and $\mathfrak{osp}_{2|2n}^{(1)}$
and their representations, and obtained character formula for the integrable highest weight representations.

In this paper, we use Wakimoto's idea of free fields to construct a family of
irreducible modules for affine Lie superalgebra $\widehat{\mathfrak{gl}_{m|n}}(\bbC)$ with infinite-dimensional weight spaces
and discuss the irreducibility of the modules in the finite and affine cases.
In fact, since the construction given here is different from Wakimoto's original constructions,
some researchers called this as Wakimoto-like construction \cite{GZ2}.
Our construction is consistent with \cite{GZ2} when we restrict $n=0$.
Moreover, The result is well known when we restrict our construction to $m=0$,
so, in this paper we always assume that $m \geq 1$.

The organization of this paper is as follows.
In Section 2 we will first construct a family of representations of the Lie superalgebras
$\mathfrak{gl}_{m|n}(\bbC)$ and $\mathfrak{sl}_{m|n}(\bbC)$,
and study the irreducibility.
Then we extend the construction to the affine Kac-Moody Lie superalgebra $\widehat{\mathfrak{gl}_{m|n}}(\bbC)$ in Section 3.
We end the paper by discussing the irreducibility of the representations.

Throughout this paper, we denote by $\bbZ$, $\bbZ_+$ and $\bbC$ the
sets of integers, nonnegative integers and complex numbers respectively.
All vector spaces and Lie (super)algebras are over $\bbC$.
For $1 \leq i \leq m+n$, the parity of $i$ is given by $|i|=0$ if $1 \leq i \leq m$, and $|i|=1$ otherwise.


\section{Representations of general linear Lie superalgebra $\mathfrak{gl}_{m|n}(\bbC)$}

Let $W_1$ be the free algebra over $\bbC$ generated by $x_i,\,y_k$,
where $i = 2,\dots, m \text{ and } k=1,\dots, n$, and $W_2$ be the two-sided ideal of $W_1$
generated by $x_i x_j-x_jx_i,\ x_iy_k-y_kx_i,\  y_k y_l+y_l y_k$, where $i,\,j=2,\dots,m\  \text{ and }\, \,k,\,l=1,\dots, n$.
Set
$$ W=W_1/{W_2}.$$

If we denote $\bbC[x_i|i=2,\ldots,m]$ to be the polynomial algebra generated by finite variables
and denote $\bigwedge^*(\bigoplus\limits_{j=1}^n\bbC y_j)$ to be the exterior algebra generated by finite variables.
Then as vector spaces we have
$$W\cong\bbC[x_i|i=2,\ldots,m]\otimes\bigwedge\nolimits^*\big(\bbC y_1 \oplus \cdots \oplus \bbC y_n\big).$$
Now we can get
\begin{align*}
&[x_i,x_j]=[x_i,y_k]=\{y_k,y_l\}_+=\{\frac{\partial}{\partial y_k},\frac{\partial}{\partial y_l}\}_+=0,\\
&[\frac{\partial}{\partial x_i},\frac{\partial}{\partial x_j}]
 =[\frac{\partial}{\partial x_i},\frac{\partial}{\partial y_k}]
 =[\frac{\partial}{\partial x_i},y_k]=[\frac{\partial}{\partial y_k},x_i]=0,\\
&[\frac{\partial}{\partial x_i},x_j]=\delta_{i,j},\{y_k,\frac{\partial}{\partial y_l}\}_+=\delta_{k,l},
\end{align*}
where $i,\,j=2,\dots m; \ k,\,l= 1,\dots, n,$
and $\frac{\partial}{\partial x_i}$,\  $\frac{\partial}{\partial y_k}$ are the partial differential operators on $W$.
Then we are ready to define the following operators on $W$:
\begin{align}
\begin{cases}
e_{1,1} \ = \  \mu-\sum\limits_{s=2}^{m} x_s \p{x_s}-\sum\limits_{t=1}^{n} y_t\p{y_t},\\
e_{i,1}=  x_i, \\
e_{m+k,1}=  y_k, \\
e_{1,i}= e_{1,1}\p{x_i}, \\
e_{1,m+k}= e_{1,1}\p{y_k}, \\
e_{i,j} = x_i\p{x_j}, \\
e_{i,m+k} =  x_i\p{y_k}, \\
e_{m+k,i} =  y_k\p{x_i}, \\
e_{m+k,m+l}=y_k\p{y_l},
\end{cases}
\end{align}
where $i,\,j=2\ldots,m$ and $k,\,l=1,\ldots,n$.
Note that $e_{1,1}+e_{2,2}+\cdot\cdot\cdot+e_{m+n,m+n}=\mu$.

We may view $W$ as a superalgebra with $|x_i| = \bar{0}$ for $i = 2,\ldots,m$ and $|y_k|= \bar{1}$ for $k = 1,\ldots,n$.
It follows that $\mathfrak{gl}(W)$ is a Lie superalgebra.
Then we have the following theorem (easily verified, but it could be viewed as a special case of Theorem 3.1,
which we will prove in Section 3):	

\begin{thm} \label{mainrep}
The linear map
\begin{displaymath}
\varphi: \mathfrak{gl}_{m|n}(\bbC) \longrightarrow \mathfrak{gl}(W)
\end{displaymath}
given by
\begin{equation*}
\varphi\big( E_{i,j} \big) \ = \ e_{i,j}, \text{ for } \, i,j = 1,\ldots, m+n,
\end{equation*}
is a Lie superalgebra homomorphism. That is, $W$ is a $\mathfrak{gl}_{m|n}(\bbC)$-module.
\end{thm}

Let $L=\mathfrak{sl}_{m|n}(\bbC)=[\mathfrak{gl}_{m|n}(\bbC),\ \mathfrak{gl}_{m|n}(\bbC)]$
and $\mathfrak{h}$ be the Cartan subalgebra of $\mathfrak{gl}_{m|n}(\bbC)$ which only contains
those diagonal matrices, then $\mathfrak{h}'=\mathfrak{h}\cap L$ is the Cartan subalgebra of $L$.
We set
\begin{align*}
\begin{cases}
&h_i=E_{i,i}-E_{i+1,i+1},\text{for}\  1\leq i\leq m-1,\\
&h_m=E_{m,m}+E_{m+1,m+1},\\
&h_j'=E_{m+j,m+j}-E_{m+j+1,m+j+1},\,\text{for}\  1\leq j\leq n-1,
\end{cases}
\end{align*}
to represent the standard basis of the Cartan subalgebra $\mathfrak{h}'$ of $\mathfrak{sl}_{m|n}(\bbC)$,
and denote by $\epsilon_i$ the basis of $\mathfrak{h}^*$ dual to $E_{i,i}$, where $1 \leq i\leq m+n$.
We can identify $\epsilon_i$ with $(-1)^{|i|}(E_{i,i},\cdot)$ with respect to
the even non-degenerated supersymmetric invariant bilinear form  $( \; , \; )$
given by $(a,b) = \mathrm{str}(ab)$ where $\mathrm{str}$ is the supertrace on $\mathfrak{gl}_{m|n}(\bbC)$.
Thus, any $\lambda\in \mathfrak{h}^*$ could be described in terms of its coordinates
as
$$\lambda=\sum_{i=1}^{m}\lambda_i\epsilon_i+\sum_{j=1}^{n}\lambda_{m+j}\epsilon_{m+j}=
(\lambda_1,...,\lambda_m|\lambda_{m+1},...,\lambda_{m+n})$$

Let us choose the standard Borel subalgebra $\mathfrak{b}\subseteq \mathfrak{gl}_{m|n}(\bbC)$ consisting of upper triangular matrices
and take a set of simple roots of $\mathfrak{gl}_{m|n}(\bbC)$ (see \cite[Chapter 1]{CW})
$$
\alpha_1=\epsilon_1-\epsilon_2,\alpha_2=\epsilon_2-\epsilon_3,\ldots, \alpha_{m+n-1}=\epsilon_{m+n-1}-\epsilon_{m+n},
$$
then the set of positive even roots and the set of positive odd roots are given by
$$
\Delta_0^+=\{\epsilon_i-\epsilon_j,\,\epsilon_{m+k}-\epsilon_{m+l}\ |1\leq i< j\leq m,\, 1\leq k<l\leq n\},
$$
$$
\Delta_1^+=\{\epsilon_i-\epsilon_{m+j}|1\leq i \leq m,\, 1\leq j\leq n\}.
$$

For $(\underline{a};\underline{b})=(a_2,\ldots,a_m;b_1,\ldots,b_n)\in\bbZ_+^{m-1}\times\{0,1\}^{n}$,
we denote $v(\underline{a};\underline{b})=x_2^{a_2}x_3^{a_3}\dots x_{m}^{a_m}y_1^{b_1}\dots y_n^{b_n}$,
and define the degree of $v(\underline{a};\underline{b})$ to be integer
$m_{\underline{a};\underline{b}}=\sum\limits_{k=2}^{m}a_k+\sum\limits_{k=1}^{n}b_k,$ i.e., $\mathrm{deg}(v(\underline{a};\underline{b}))=m_{\underline{a};\underline{b}}.$
So the weight of $v(\underline{a};\underline{b})$ is
$$(\mu-m_{\underline{a};\underline{b}},a_2,\dots,a_m|b_{1},\dots, b_{n}).$$

We have the following lemma.

\begin{lemma} \label{mneqn}
$\mathfrak{h}'$ acts diagonally on $W$ (so $W$ is a weight module over $\mathfrak{sl}_{m|n}(\bbC)$). If $m\ne n$ and $(\underline{a};\underline{b})\ne (\underline{c};\underline{d})$, then $v(\underline{a};\underline{b})$ and $v(\underline{c};\underline{d})$ belong to different weight spaces.
\end{lemma}

\begin{proof} Choose $\{E_{1,1}-E_{i,i} |\, i=2,\dots, m \} \cup \{E_{1,1}+E_{m+k,m+k}|\, k=1,\dots, n \}$
as a basis of $\mathfrak{h}'$. Then
\begin{align*}
(e_{1,1}-e_{i,i}). v(\underline{a};\underline{b})=(\mu-a_i-m_{\underline{a};\underline{b}})v(\underline{a};\underline{b})
\end{align*}
and
\begin{align*}
(e_{1,1}+e_{m+k,m+k}). v(\underline{a};\underline{b})=
(\mu+b_k-m_{\underline{a};\underline{b}})v(\underline{a};\underline{b}).
\end{align*}
If $(\underline{a};\underline{b}), (\underline{c};\underline{d}) \in \bbZ_+^{m-1}\times\{0,1\}^n$
and $v(\underline{a};\underline{b})$, $v(\underline{c};\underline{d})$ lie in the same weight space,
then we have
\begin{equation} \label{relat1}
\mu -\sum\limits_{s=2}^{m} a_s-\sum\limits_{t=1}^{n} b_t -a_i = \mu -\sum\limits_{s=2}^{m} c_s-\sum\limits_{t=1}^{n} d_t -c_i,
\end{equation}
\begin{equation} \label{relat2}
\mu -\sum\limits_{s=2}^{m} a_s-\sum\limits_{t=1}^{n} b_t +b_k = \mu -\sum\limits_{s=2}^{m} c_s-\sum\limits_{t=1}^{n} d_t +d_k,
\end{equation}
for $i=2,3,\cdots,m$ and $k=1,\ldots,n$,
which imply
\begin{align*}
-m\sum_{s=2}^{m} a_s-(m-1)\sum_{t=1}^{n} b_t &= -m\sum\limits_{s=2}^{m} c_s-(m-1)\sum\limits_{t=1}^{n} d_t,\\
-n\sum_{s=2}^{m} a_s-(n-1)\sum_{t=1}^{n} b_t &= -n\sum\limits_{s=2}^{m} c_s-(n-1)\sum\limits_{t=1}^{n} d_t.
\end{align*}
Thus
\begin{align*}
(n-m)m_{\underline{a};\underline{b}}=(n-m)m_{\underline{c};\underline{d}}.
\end{align*}
If $m\neq n$, we have $m_{\underline{a};\underline{b}}=m_{\underline{c};\underline{d}}$. By relations \eqref{relat1} and \eqref{relat2},
we have $(\underline{a};\underline{b}) = (\underline{c};\underline{d})$.
\end{proof}

\begin{lemma} \label{meqn}
\begin{enumerate}
\item
If $m=n=1$, then $W=\bigwedge\bbC y_1$ is the two dimensional weight $\mathfrak{sl}_{1|1}(\bbC)$-module.
Moreover, if $\mu\ne0$, then $W$ is irreducible, or otherwise $\bbC y_1$ is a submodule of $W$.

\item
For $m=n\geq2$, $x_2^{a_2+1}x_3^{a_3+1}\ldots x_m^{a_m+1}$ and
$x_2^{a_2}x_3^{a_3}\ldots x_m^{a_m}y_1y_2\ldots y_m$ have the same weight for $a_i \in \bbZ_+$ with $2 \leq i \leq m$.
moreover, two monic monomials $f_1,\,f_2 \in W$
lie in the same weight space if and only if
$\{f_1,\,f_2\}=\{x_2^{a_2+1}x_3^{a_3+1}\ldots x_m^{a_m+1},\,x_2^{a_2}x_3^{a_3}\ldots x_m^{a_m}y_1y_2\ldots y_m\}$
for some $a_i \in \bbZ_+$, $2 \leq i \leq m$.
\end{enumerate}
\end{lemma}

\begin{proof}
The first one is obvious.

If $m=n=2$, then $W=\bbC[x_2]\bigotimes\bigwedge^*(\bbC y_1\oplus\bbC y_2)$.
We can split monic monomials in $W$ into $4$ different classes with respect to $y_k$,
$$x_2^{a_2},\,\,\,x_2^{b_2}y_1,\,\,\,x_2^{c_2}y_2,\,\,\,x_2^{d_2}y_1y_2,$$
for nonnegative integers $a_2,\ b_2\ ,c_2,\ d_2$.
The weights of these monomials are
$$
(\mu-2a_2,a_2,0),\,(\mu-2b_2-1,b_2+1,1),\,(\mu-2c_2-1,c_2,-1),\,(\mu-2d_2-2,d_2+1,0).
$$
Since $f_1,\,f_2$ belongs to the same weight space, we know that the possible cases are
$f_1=x_2^{a_2+1}$ and $f_2=x_2^{a_2}y_1y_2$, or reverse.

For general case, we can split monic monomials in $W$ into $2^n$ different classes with respect to $y_k$,
and then using the same arguments as in the $m=n=2$ case, we can prove only
$x_2^{a_2+1}x_3^{a_3+1}\ldots x_m^{a_m+1}$ and $x_2^{a_2}x_3^{a_3}\ldots x_m^{a_m}y_1y_2\ldots y_m$
belong to the same weight space for arbitrary $a_2,\,a_3,\,\ldots,a_m \in \bbZ_+$.
\end{proof}

\begin{thm} \label{finite}
\begin{enumerate}
\item
If $m\geq2$, then $W$ is the highest weight module of $\mathfrak{sl}_{m|n}(\bbC)$,
with the highest weight vector 1, and the highest weight $\lambda$ is
$$
(\lambda(h_1),\cdots,\lambda(h_{m}),\lambda(h_1'),\cdots, \lambda(h_{n-1}'))=(\mu,0,\dots,0).
$$
Moreover, $W$ is an irreducible module if and only if $\mu$ is not a nonnegative integer.

\item
If $m=1$ and $n \geq 2$, then $W$ is the highest weight module of $\mathfrak{sl}_{1|n}(\bbC)$,
with the highest weight vector 1, and the highest weight $\lambda$ is
$$
(\lambda(h_1),\lambda(h'_1),\lambda(h'_2),\cdots,\lambda(h'_{n-1}))=(\mu,0,\cdots,0).
$$
Moreover, $W$ is an irreducible module if and only if $\mu \notin \{0,1,\cdots,n-1\}$.
\end{enumerate}
\end{thm}
\begin{proof}(1). It is obvious that $W$ is a highest weight module, with highest weight vector 1 and the highest weight $\lambda=(\mu,0,\cdots,0)$.

If $m\ne n$ and $U$ is a submodule of $W$, we may assume that $v(\underline{a};\underline{b})$ is
in $U$ with the degree $s=m_{\underline{a};\underline{b}}\geq 1$ since every monomial lies in different weight space.
Note that
\begin{align*}
x_2^{s}=\frac{1}{a_3!\cdots a_m!}e_{2,m+n}^{b_n}\cdots e_{2,m+1}^{b_1}e_{2,m}^{a_m}\cdots e_{2,3}^{a_3}v(\underline{a};\underline{b})
\end{align*}
hence $x_2^{s}$ is in $U.$ Since
$$
e_{1,2}(x_2^s)=[s\mu-s(s-1)]x_2^{s-1}=s(\mu-(s-1))x_2^{s-1},
$$
we have
$$
e_{1,2}^s(x_2^s)=s!(\mu-(s-1))\cdots(\mu-1)\mu \cdot 1.
$$
So, if $\mu$ is not a nonnegative integer, then $1\in U$. Hence $W$ is an irreducible module.

For $m=n\geq 2,$ by Lemma \ref{meqn}, we just consider nonzero weight vector
$$
v=A x_2^{a_2+1}x_3^{a_3+1}\ldots x_m^{a_m+1}+B x_2^{a_2}x_3^{a_3}\ldots x_m^{a_m}y_1y_2\ldots y_m\in U
$$
with $A,B \in \bbC$ and $a_i \in \bbZ_+$. If $B=0$, then $x_2^{a_2+1}x_3^{a_3+1}\ldots x_m^{a_m+1} \in U$,
otherwise
$$
e_{s,m+m}\ldots e_{s,m+2}e_{2,m+1}(B^{-1}v)=x_2^{a_2+m}x_3^{a_3}\ldots x_m^{a_m} \in U.
$$
By the same arguments as in the $m \ne n$ case, we know that
$W$ is an irreducible module if and only if $\mu$ is not a nonnegative integer.

(2). Since $(e_{1,1}+e_{2,2})1=\mu,\text{ and
}(e_{i,i}-e_{i+1,i+1})1=0,\,i=2,3,\cdots,n$, $W$ is a highest weight
module with highest weight vector
1 and the highest weight $\lambda=(\mu,0,\cdots,0)$.

If $U$ is a nonzero submodule of $W$, such that
$y_{i_1}y_{i_2}\cdots y_{i_k},1 \leq i_1 < i_2 < \cdots < i_k \leq
n$ is in $U$ (since every monomial lies in different weight space).
We have
\begin{align*}
  & e_{1,i_k}\cdots e_{1,i_2}e_{1,i_1}y_{i_1}y_{i_2}\cdots y_{i_k} \\
= & e_{1,i_k}\cdots e_{1,i_2}e_{1,1}y_{i_2}\cdots y_{i_k}   \\
= & (\mu -(k-1))e_{1,i_k}\cdots e_{1,i_2}y_{i_2}\cdots y_{i_k}  \\
= & \cdots \cdots \\
= & (\mu -(k-1))(\mu -(k-2)) \cdots (\mu -1) \mu \cdot 1.
\end{align*}
So, if $\mu \notin \{0,1,\cdots,n-1\}$, then 1 is in $U$, hence $W$
is an irreducible module.
\end{proof}

\begin{thm}
\begin{enumerate}
\item
If $m\geq2$ and $\mu$ is a nonnegative integer,
the submodule $V(\mu)$ of $W$ generated by $x_2^{\mu+1}$ is the unique maximal submodule of $W$.
Moreover $V(\mu)$ itself is irreducible.

\item
If $m=1$ and $\mu \in \{0,1,\cdots,n-1\}$,
the submodule $V(\mu)$ of $W$ generated by $y_1y_2 \cdots y_{\mu+1}$ is the unique maximal submodule of $W$,
also itself is irreducible.
\end{enumerate}
\end{thm}

\begin{proof}
(1).
If $U$ is a submodule of $W$ containing $V(\mu)$ as a proper submodule,
then $U$ must contain some $v(\underline{a};\underline{b})$ with degree $s:=m_{\underline{a};\underline{b}}\leq \mu$,
which follows from Lemma \ref{mneqn}, Lemma \ref{meqn} and the proof of Theorem \ref{finite}.
Hence
\begin{align*}
&x_2^{s}=\frac{1}{a_3!\cdots a_m! } e_{2,m+n}^{b_n}
\cdots e_{2,m+1}^{b_1}e_{2,m}^{a_m}\cdots e_{2,3}^{a_3}v(\underline{a};\underline{b}) \in U,
\end{align*}
furthermore
$$
1=\frac{1}{s!(\mu-(s-1))\cdots(\mu-1)\mu}e_{1,2}^s(x_2^s) \in U.
$$
So $U=W$.

It is obvious that the monic singular vector in $V(\mu)$ must have the form $x_2^s$ with $s \geq \mu+1$.
Now
$$
0=e_{1,2}(x_2^s)=[s\mu-s(s-1)]x_2^{s-1}=s(\mu-(s-1))x_2^{s-1}
$$
implies that $s=\mu+1$, i.e., $V(\mu)$ is an irreducible module.
Now the uniqueness of $V(\mu)$ follows from the above discussion and the proof of Theorem \ref{finite}.

(2). If $U$ is a submodule of $W$ containing $V(\mu)$ as
its submodule, then $U$ must contain $y_{i_1}y_{i_2}\cdots y_{i_s}$
with $0 \leq s \leq \mu$ and $1 \leq i_1 < \cdots i_s \leq n$.
Hence we have
$$
1=\frac{1}{(\mu-(s-1))(\mu-(s-2)) \cdots (\mu -1)\mu}e_{1,i_s}\cdots e_{1,i_1}y_{i_1}\cdots y_{i_s} \in U.
$$
So $U=W$, i.e., $V(\mu)$ is maximal.

Also if $W$ contains another monic singular vector different from $1$.
It is obvious that the singular vector must have the form $y_1y_2 \cdots y_s$ with $s \geq 1$.
Now $e_{1,2}y_1y_2 \cdots y_s=(\mu-(s-1))y_2 \cdots y_s = 0$ deduce that $s = \mu+1$.
Hence $V(\mu)$ is the unique maximal submodule of $W$, moreover itself is irreducible.
\end{proof}

Next let $S$ be the subalgebra of $\mathfrak{gl}(W)$ generated by elements $e_{i,j}$ with $i,\,j=2,\ldots,m+n$.
Then we know that $S \cong \mathfrak{gl}_{m-1|n}(\bbC)$ which follows from Theorem \ref{mainrep}.
Thus $\mathfrak{h}_S=\mathfrak{h}\cap S$ is the Cartan subalgebra of $S$,
and $\epsilon_i$, $2\leq i\leq m+n-1$ is a basis for $(\mathfrak{h}_S)^*$.

When $\mu$ is a nonnegative integer,
evidently $V(\mu)$ is the submodule of $W$ spanned by all the monomials $v(\underline{a};\underline{b})$,
with degree  $m_{\underline{a};\underline{b}}\geq \mu+1$.
Now the quotient module $W/V(\mu)$ is an irreducible highest weight module of $\mathfrak{sl}_{m|n}(\bbC)$,
with highest weight $(\mu,0,\cdots,0)$, hence it is finite-dimensional
(for more information about the finite-dimensional simple modules over $\mathfrak{sl}_{m|n}(\bbC)$, see \cite{CW}).
It is clear that
$$\{v(\underline{a};\underline{b})+V(\mu)\ |\ m_{\underline{a};\underline{b}}\leq\mu\}$$
is a basis for $W/V(\mu)$ and the weight of $v(\underline{a};\underline{b})+V(\mu)$ is
$$
\mu \epsilon_1-m_{\underline{a};\underline{b}}\alpha_1- (m_{\underline{a};\underline{b}}-a_2)\alpha_2-
(m_{\underline{a};\underline{b}}-a_2-a_3)\alpha_3-\ldots-b_n\alpha_{m+n-1},
$$
so the character formula for $M/V(\mu)$ is
$$
Ch_{W/V(\mu)}=\sum_{\stackrel{a_i \geq0,\, 0 \leq b_k\leq1}{m_{\underline{a};\underline{b}}\leq\mu}}
e(\mu \epsilon_1)e(-(m_{\underline{a};\underline{b}})\alpha_1-\ldots-b_n\alpha_{m+n-1}).
$$

For $s\in\bbZ_+$, set $W_s = \{f\in W | f \text{ is a homogenous elements of degree } s\}$, then we get
$$
W=\bigoplus\limits_{s=0}^{\infty}W_s.
$$
From the above discussion, we can obtain the following theorem:

\begin{thm}
\begin{enumerate}
\item
The character formula of $\mathfrak{gl}_{m|n}(\bbC)$-module $W$ is
$$
Ch(W)=\sum\limits_{\alpha\in P}e(\mu \epsilon_1)e(\alpha),
$$
where
$$
P=\Big\{\sum\limits_{i=1}^{m+n-1}-t_i\alpha_i \, \Big | \,t_1\geq t_2\geq \ldots\geq t_{m+n-1}\geq t_{m+n} = 0,
\,t_j-t_{j+1} \in \{0,1\} \text{ for } m\leq j\leq m+n-1\Big\}.
$$

\item
If $\mu$ is a nonnegative integer, the character formula of $V(\mu)$ is
$$
Ch\big(V(\mu)\big)=\sum\limits_{\alpha\in Q}e(\mu \epsilon_1)e(\alpha),
$$
where
$$
Q=\Big\{\sum\limits_{i=1}^{m+n-1}-t_i\alpha_i \in P\, \Big|\,t_1\geq\mu+1 \Big\}.
$$

\item
If $m\geq2$, then every homogeneous subspace $W_s$ of $W$ is an irreducible highest weight module of $S$,
with the highest weight vector $x_2^s$ of weight $s\epsilon_2$.

\item
If $m=1$, then every homogenous subspace $W_s$ (with $s \in \{0,1, \cdots, n\}$) of $W$ is an irreducible
highest weight module of $S$, with the highest weight vector is $y_1 y_2 \cdots y_{s}$ of weight $\epsilon_2+\epsilon_3+\cdots+\epsilon_{s+1}$.
\end{enumerate}
\end{thm}

\section{Representations of affine Lie superalgebra $\widehat{\mathfrak{gl}_{m|n}}(\bbC)$}


In this section we focus on some affine Lie superalgebra. Let
$$
\widehat{\mathfrak{gl}_{m|n}}(\bbC)=\widetilde{\mathfrak{gl}_{m|n}}(\bbC)\oplus\bbC d
$$
be the affine Kac-Moody Lie superalgebra associated to $\mathfrak{gl}_{m|n}(\bbC)$ (see \cite{CS}), here
$$\widetilde{\mathfrak{gl}_{m|n}}(\bbC)=\mathfrak{gl}_{m|n}(\bbC)\otimes\bbC [t,t^{-1}]\oplus\bbC K ,$$
and $K$ is a central clement of $\widehat{\mathfrak{gl}_{m|n}}(\bbC)$.
The commutation relations in $\widehat{\mathfrak{gl}_{m|n}}(\bbC)$ are given by
$$
[a\otimes t^{m_1}, b\otimes t^{n_1}] = [a, b]\otimes t^{m_1+n_1} + m_1 \delta_{m_1,-n_1}(a, b)K,\;\;
[d, a\otimes t^{m_1}] = m_1 a\otimes t^{m_1},\;\; [K,\widehat{\mathfrak{gl}_{m|n}}(\bbC)] = 0.
$$
where $(\;.\;)$ is the non-degenerate even invariant supersymmetric bilinear form on $\mathfrak{gl}_{m|n}(\bbC)$.
In particular, we can regard $\mathfrak{gl}_{m|n}(\bbC)$ as the subalgebra $\mathfrak{gl}_{m|n}(\bbC) \otimes 1$
of $\widehat{\mathfrak{gl}_{m|n}}(\bbC)$ and
$H=\mathfrak{h}\otimes1\oplus\bbC K \oplus\bbC d$ is the Cartan subalgebra of $\widehat{\mathfrak{gl}_{m|n}}(\bbC)$.

Let $\widehat{W}_1$ be the free algebra generated by infinite elements
$x_i(m_1),\, y_k(n_1)$ for $m_1,\,n_1\in\bbZ$, $i=2,\ldots,m$ and $k=1,\ldots,n$.
Take $\widehat{W}_2$ to be the two-sided ideal of $\widehat{W}_1$ generated by
$$
x_i(m_1)x_j(n_1)-x_j(n_1)x_i(m_1),\,\,
x_i(m_1)y_k(n_1)-y_k(n_1)x_i(m_1),\,\,
y_k(m_1)y_l(n_1)+y_l(n_1)y_k(m_1)
$$
for $m_1,\,n_1\in\bbZ$, $i,\,j=2,\ldots,m$ and $k,\,l=1,\ldots,n$.
Set
$$
\widehat{W}=\widehat{W}_1/{\widehat{W}_2}.
$$

Then we have
\begin{align*}
&[x_i(m_1),x_j(n_1)]=[x_i(m_1),y_k(n_1)]=\{y_k(m_1),y_l(n_1)\}_+=\{\frac{\partial}{\partial y_k(m_1)},\frac{\partial}{\partial y_l(n_1)}\}_+=0,\\
&[\frac{\partial}{\partial x_i(m_1)},\frac{\partial}{\partial x_j(n_1)}]=[\frac{\partial}{\partial x_i(m_1)},\frac{\partial}{\partial y_k(n_1)}]=[\frac{\partial}{\partial x_i(m_1)},y_k(n_1)]=[\frac{\partial}{\partial y_k(m_1)},x_i(n_1)]=0,\\
&[\frac{\partial}{\partial x_i(m_1)},x_j(n_1)]=\delta_{i,j}\delta_{m_1,n_1},\,\{y_k(m_1),\frac{\partial}{\partial y_l(n_1)}\}_+=\delta_{k,l}\delta_{m_1,n_1},
\end{align*}
for $m_1,\,n_1\in\bbZ$, $i,\,j=2,\ldots,m$ and $k,\,l=1,\ldots,n$.
We define the following operators on $\widehat{W}$:
\begin{align}
\begin{cases}
& e_{1,1}(m_1) \ = \  \mu\delta_{m_1,0}-\sum\limits_{m_2\in\bbZ}\sum\limits_{i=2}^{m} x_i(m_1+m_2) \frac{\partial} {\partial x_i(m_2)}-\sum\limits_{m_2\in\bbZ}\sum\limits_{i=1}^{n} y_i(m_1+m_2)\frac{\partial}{\partial y_i(m_2)},\\
&e_{1,i}(m_1)=\sum\limits_{m_2\in\bbZ}e_{1,1}(m_1+m_2)\frac{\partial}{\partial x_i(m_2)},\\
&e_{1,m+k}(m_1)=\sum\limits_{m_2\in\bbZ}e_{1,1}(m_1+m_2)\frac{\partial}{\partial y_k(m_2)},\\
&e_{i,1}(m_1)=x_i(m_1), \\
&e_{m+k,1}(m_1)=y_k(m_1),\\
&e_{i,j}(m_1)=\sum\limits_{m_2\in \bbZ}x_i(m_1+m_2)\frac{\partial}{\partial x_j(m_2)},\\
&e_{i,m+k}(m_1)=\sum\limits_{m_2\in \bbZ}x_i(m_1+m_2)\frac{\partial}{\partial y_i(m_2)}, \\
&e_{m+k,j}(m_1)=\sum\limits_{m_2\in \bbZ}y_k(m_1+m_2)\frac{\partial}{\partial x_j(m_2)}, \\
&e_{m+k,m+l}(m_1)=\sum\limits_{m_2\in \bbZ}y_k(m_1+m_2)\frac{\partial}{\partial y_l(m_2)},
\end{cases}
\end{align}
where $m_1\in\bbZ$, $i,\,j=2,\ldots,m$ and $k,\,l=1,\ldots,n$.

We may view $\widehat{W}$ as a superalgebra with $|x_i(m_1)| = \bar{0}$,
$|y_k(m_1)| = \bar{1}$ for $m_1\in\bbZ$, $i=2,\ldots,m$ and $k=1,\ldots,n$.
It follows that $\mathfrak{gl}(\widehat{W})$ is a Lie superalgebra. Then we have the following theorem:	

\begin{thm} \label{mainrep1}
\begin{enumerate}
The linear map
\begin{displaymath}
\varphi: \widetilde{\mathfrak{gl}_{m|n}}(\bbC) \longrightarrow \mathfrak{gl}(\widehat{W})
\end{displaymath}
given by
\begin{equation*}
\varphi\big( E_{i,j} \otimes t^{m_1} \big) \ = \ e_{i,j}(m_1), \text{ for } \, i,j = 1,\ldots, m+n, \quad \varphi(K)=0,
\end{equation*}
is a Lie superalgebra homomorphism.  That is, $\widehat{W}$ is a $\widetilde{\mathfrak{gl}_{m|n}}(\bbC)$-module.
\end{enumerate}
\end{thm}
\begin{proof}
It suffices to show that the operators given above satisfy the corresponding commutator relations.
We will check these case by case
\begin{align*}
\big[ e_{1,1}(m_1), & e_{1,1}(n_1) \big] \\
= & \bigg[\summn \sum_{i=2}^{m}  x_i(m_1+m_2) \p{x_i(m_2)}, \summpnp \sum_{j=2}^{m}  x_j(n_1+n_2) \p{x_j(n_2)} \bigg] \\
& +\bigg[\summn \sum_{i=2}^{m}  x_i(m_1+m_2) \p{x_i(m_2)}, \summpnp \sum_{j=1}^{n}  y_j(n_1+n_2) \p{y_j(n_2)} \bigg] \\
& +\bigg[\summn \sum_{i=1}^{n}  y_i(m_1+m_2) \p{y_i(m_2)}, \summpnp \sum_{j=2}^{m}  x_j(n_1+n_2) \p{x_j(n_2)} \bigg] \\
& +\bigg[\summn \sum_{i=1}^{n}  y_i(m_1+m_2) \p{y_i(m_2)}, \summpnp \sum_{j=1}^{n}  y_j(n_1+n_2) \p{y_j(n_2)} \bigg] \\
= & \summnmpnp \sum_{i,j=2}^{m} \bigg( x_i(m_1+m_2) \Big[\p{x_i(m_2)},  x_j(n_1+n_2)\Big] \p{x_j(n_2)}\\
& -x_j(n_1+n_2)\Big[\p{x_j(n_2)},  x_i(m_1+m_2)\Big] \p{x_i(m_2)}\bigg) \\
&  +\summnmpnp \sum_{i,j=1}^{n} \bigg( y_i(m_1+m_2) \Big\{\p{y_i(m_2)},  y_j(n_1+n_2)\Big\}_+ \p{y_j(n_2)}\\
& -y_j(n_1+n_2)\Big\{\p{y_j(n_2)},  y_i(m_1+m_2)\Big\}_+ \p{y_i(m_2)}\bigg) \\
= & \summnmpnp \sum_{i,j=2}^{m}  \bigg(\delta_{i,j}\delta_{m_2,n_1+n_2}x_i(m_1+m_2)\p{x_j(n_2)}
           -\delta_{i,j}\delta_{n_2,m_1+m_2}x_j(n_1+n_2)\p{x_i(m_2)}\bigg)\\
&+\summnmpnp \sum_{i,j=1}^{n} \bigg(\delta_{i,j}\delta_{m_2,n_1+n_2}y_i(m_1+m_2)\p{y_j(n_2)}
           -\delta_{i,j}\delta_{n_2,m_1+m_2}y_j(n_1+n_2)\p{y_i(m_2)}\bigg)\\
= & 0
\end{align*}

\begin{align*}
\big[ e_{1,1}(m_1), \ e_{1,i}(n_1) \big]
= & \summn  \Bigg( \Big[ e_{1,1}(m_1), e_{1,1}(n_1+m_2) \Big] \p{x_i(m_2)}\\
&+ e_{1,1}(n_1+m_2) \bigg[- \summpnp \sum_{j=2}^{m}  x_j(m_1+n_2) \p{x_j(n_2)}, \p{x_i(m_2)} \bigg]\Bigg)\\
= & \summnmpnp \sum_{j=2}^{m} e_{1,1}(n_1+m_2) \bigg[ \p{x_i(m_2)},x_j(m_1+n_2)\bigg]
 \p{x_j(n_2)} \\
= & e_{1,i}(m_1+n_1),
\end{align*}
for $i=2,\ldots,m$.
\begin{align*}
\big[ e_{1,1}(m_1), \ e_{1,m+i}(n_1) \big]
= & \summn  \Bigg( \big[ e_{1,1}(m_1), e_{1,1}(n_1+m_2) \big] \p{y_i(m_2)}\\
& + e_{1,1}(n_1+m_2) \bigg[- \summpnp \sum_{j=1}^{n}  y_j(m_1+n_2) \p{y_j(n_2)}, \p{y_i(m_2)} \bigg]\Bigg)\\
= & \summnmpnp \sum_{j=1}^{n} e_{1,1}(n_1+m_2) \bigg\{y_j(m_1+n_2), \p{y_i(m_2)}\bigg\}_+ \p{y_j(n_2)} \\
= & e_{1,m+i}(m_1+n_1),
\end{align*}
for $i=1,\ldots,n$.
\begin{align*}
\big[ e_{i,1}(m_1), & \ e_{1,1}(n_1) \big] \\
= & - \sum_{m_2\in\bbZ} \sum_{j=2}^{m} \bigg[ x_i(m_1),\  x_j(n_1+m_2) \p{x_j(m_2)} \bigg]
    - \sum_{m_2\in\bbZ} \sum_{k=1}^{n} \bigg[ x_i(m_1),\  y_k(n_1+m_2) \p{y_k(m_2)} \bigg]\\
= & - \sum_{m_2\in\bbZ} \sum_{j=2}^m  x_j(n_1+m_2) \bigg[ x_i(m_1), \p{x_j(m_2)} \bigg]\\
= & \summn \sum_{j=2}^m  x_j(n_1+m_2) \delta_{i,j} \delta_{m_1,m_2} \\
= & e_{i,1}(m_1+n_1),
\end{align*}
for $i=2,\ldots,m$.
\begin{align*}
\big[ e_{m+i,1}(m_1), \ e_{1,1}(n_1) \big]
= & - \sum_{m_2\in\bbZ} \sum_{j=1}^{n} \bigg[ y_i(m_1),\  y_j(n_1+m_2) \p{y_j(m_2)} \bigg] \\
= &  \sum_{m_2\in\bbZ} \sum_{j=1}^{n}  y_j(n_1+m_2) \bigg\{ y_i(m_1), \p{y_j(m_2)} \bigg\}_+  \\
= & \sum_{m_2\in\bbZ} \sum_{j=1}^{n}  y_j(n_1+m_2) \delta_{i,j} \delta_{m_1,m_2} \\
= & e_{m+i,1}(m_1+n_1),
\end{align*}
for $i=1,\ldots,n$.
\begin{align*}
\big[ e_{i,j}(m_1), \ e_{1,1}(n_1) \big]
= &- \bigg[  \summn x_i(m_1+m_2) \p{x_j(m_2)}, \ \summpnp \sum_{k=2}^{m}  x_k(n_1+n_2) \p{x_k(n_2)} \bigg] \\
= & \summnmpnp \ \sum_{k=2}^{m} \Bigg( - x_i(m_1+m_2) \bigg[ \p{x_j(m_2)}, x_k(n_1+n_2) \bigg] \p{x_k(n_2)} \\
& + x_k(n_1+n_2) \bigg[ \p{x_k(n_2)}, x_i(m_1+m_2) \bigg] \p{x_j(m_2)} \Bigg)\\
= & - \summpnp  x_i(m_1+n_1+n_2) \p{x_k(n_2)} + \summn  x_i(m_1+n_1+m_2) \p{x_j(m_2)} \\
= & 0,
\end{align*}
for $i,\,j=2,\ldots,m$.
\begin{align*}
\big[ e_{i,m+j}(m_1), \ e_{1,1}(n_1) \big]
= & -\bigg[  \summn x_i(m_1+m_2) \p{y_j(m_2)}, \ \summpnp \sum_{k=2}^{m}  x_k(n_1+n_2) \p{x_k(n_2)} \bigg] \\
& - \bigg[  \summn x_i(m_1+m_2) \p{y_j(m_2)}, \ \summpnp \sum_{k=1}^{n}  y_k(n_1+n_2) \p{y_k(n_2)} \bigg] \\
= & \summnmpnp \ \sum_{k=2}^{m} x_k(n_1+n_2) \bigg[ \p{x_k(n_2)}, x_i(m_1+m_2) \bigg] \p{y_j(m_2)} \\
& - \summnmpnp \ \sum_{k=1}^{n}  x_i(m_1+m_2) \bigg\{ \p{y_j(m_2)}, y_k(n_1+n_2) \bigg\}_+  \p{y_k(n_2)} \\
= & \summn  x_i(m_1+n_1+m_2) \p{y_j(m_2)}- \summpnp  x_i(m_1+n_1+n_2) \p{y_j(n_2)}  \\
= & 0,
\end{align*}
for $i=2,\ldots,m$ and $j=1,\ldots,n$.
\begin{align*}
\big[ e_{m+i,j}(m_1), \ e_{1,1}(n_1) \big]
= &- \bigg[  \summn y_i(m_1+m_2) \p{x_j(m_2)}, \ \summpnp \sum_{k=2}^{m}  x_k(n_1+n_2) \p{x_k(n_2)} \bigg] \\
& - \bigg[  \summn y_i(m_1+m_2) \p{x_j(m_2)}, \ \summpnp \sum_{k=1}^{n}  y_k(n_1+n_2) \p{y_k(n_2)} \bigg] \\
= & -\summnmpnp \ \sum_{k=2}^{m} y_i(m_1+m_2) \bigg[ \p{x_j(m_2)}, x_k(n_1+n_2) \bigg]\p{x_k(n_2)} \\
& + \summnmpnp \ \sum_{k=1}^{n} y_k(n_1+n_2) \bigg\{ \p{y_k(n_2)}, y_i(m_1+m_2) \bigg\}_+  \p{x_j(m_2)}\\
= & - \summpnp  y_i(m_1+n_1+n_2) \p{x_j(n_2)} + \summn  y_i(m_1+n_1+m_2) \p{x_j(m_2)} \\
= & 0,
\end{align*}
for $i=1,\ldots,n$ and $j=2,\ldots,m$.
\begin{align*}
\big[ e_{m+i,m+j}(m_1), \ e_{1,1}(n_1) \big]
= &- \bigg[  \summn y_i(m_1+m_2) \p{y_j(m_2)}, \ \summpnp \sum_{k=2}^{m}  x_k(n_1+n_2) \p{x_k(n_2)} \bigg] \\
& - \bigg[  \summn y_i(m_1+m_2) \p{y_j(m_2)}, \ \summpnp \sum_{k=1}^{n}  y_k(n_1+n_2) \p{y_k(n_2)} \bigg] \\
= & \summnmpnp \ \sum_{k=1}^{n} \Bigg( - y_i(m_1+m_2) \bigg\{ \p{y_j(m_2)}, y_k(n_1+n_2) \bigg\}_+  \p{y_k(n_2)} \\
& + y_k(n_1+n_2) \bigg\{ \p{y_k(n_2)}, y_i(m_1+m_2) \bigg\}_+ \p{y_j(m_2)} \Bigg)\\
= & - \summpnp  y_i(m_1+n_1+n_2) \p{y_j(n_2)} + \summn  y_i(m_1+n_1+m_2) \p{y_j(m_2)} \\
= & 0,
\end{align*}
for $i,\,j=1,\ldots,n$.
\begin{align*}
\big[ e_{1,i}(m_1), & \ e_{1,j}(n_1) \big] \\
= & \summn  \big[ e_{1,i}(m_1), \ e_{1,1}(n_1+m_2) \big] \p{x_j(m_2)}\\
& + \summn  e_{1,1}(n_1+m_2) \bigg[ \summpnp  e_{1,1}(m_1+n_2) \p{x_i(n_2)}, \p{x_j(m_2)} \bigg] \\
= &  -\summn e_{1,i}(m_1+n_1+m_2)\p{x_j(m_2)}\\
& - \summnmpnp  e_{1,1}(n_1+m_2) \Bigg[ \bigg( \summppnpp \sum_{k=2}^{m}  x_k(m_1+n_2+n_3) \p{x_k(n_3)} \bigg) \p{x_i(n_2)}, \p{x_j(m_2)} \Bigg] \\
= & -\summn e_{1,i}(m_1+n_1+m_2)\p{x_j(m_2)}\\
&  - \summnmpnpmppnpp \sum_{k=2}^{m}  e_{1,1}(n_1+m_2) \bigg[ x_k(m_1+n_2+n_3), \p{x_j(m_2)}\bigg] \p{x_k(n_3)} \p{x_i(n_2)}\\
= & -\summn e_{1,i}(m_1+n_1+m_2)\p{x_j(m_2)}  + \summpnpmppnpp  e_{1,1}(n_1+m_1+n_2+n_3) \p{x_j(n_3)} \p{x_i(n_2)} \\
= & 0,
\end{align*}
for $i,\,j=2,\ldots,m$.
\begin{align*}
\big[ e_{1,i}&(m_1), \ e_{1,m+j}(n_1) \big] \\
= & \summn  \big[ e_{1,i}(m_1), \ e_{1,1}(n_1+m_2) \big] \p{y_j(m_2)}
    + \summn  e_{1,1}(n_1+m_2) \bigg[ \summpnp  e_{1,1}(m_1+n_2) \p{x_i(n_2)}, \p{y_j(m_2)} \bigg] \\
= &  -\summn e_{1,i}(m_1+n_1+m_2)\p{y_j(m_2)} \\
& - \summnmpnp  e_{1,1}(n_1+m_2) \Bigg[ \bigg( \summppnpp \sum_{k=1}^{n}  y_k(m_1+n_2+n_3) \p{y_k(n_3)} \bigg) \p{x_i(n_2)}, \p{y_j(m_2)} \Bigg] \\
= & -\summn e_{1,i}(m_1+n_1+m_2)\p{y_j(m_2)}\\
& + \summnmpnpmppnpp \sum_{k=1}^{n}  e_{1,1}(n_1+m_2) \bigg\{ y_k(m_1+n_2+n_3), \p{y_j(m_2)}\bigg\}_+ \p{y_k(n_3)} \p{x_i(n_2)}\\
= & -\summn e_{1,i}(m_1+n_1+m_2)\p{y_j(m_2)}  + \summpnpmppnpp  e_{1,1}(n_1+m_1+n_2+n_3) \p{y_j(n_3)} \p{x_i(n_2)} \\
= & -\summn e_{1,i}(m_1+n_1+m_2)\p{y_j(m_2)}  + \summppnpp e_{1,i}(m_1+n_1+n_3)\p{y_j(n_3)}\\
= & 0,
\end{align*}
for $i=2,\ldots,m$ and $j=1,\ldots,n$.
\begin{align*}
\big[ e_{i,1}(m_1), e_{1,j}(n_1) \big]
= & \bigg[ e_{i,1}(m_1), \ \sum_{m_2\in\bbZ}  e_{1,1}(n_1+m_2) \p{x_j(m_2)}   \bigg] \\
= & \sum_{m_2\in\bbZ}  \Bigg( \bigg[ e_{i,1}(m_1), e_{1,1}(n_1+m_2) \bigg] \p{x_j(m_2)} + e_{1,1}(n_1+m_2) \bigg[ x_i(m_1), \p{x_j(m_2)} \bigg] \Bigg) \\
= &  e_{i,j}(m_1+n_1) - \sum_{m_2\in\bbZ} e_{1,1}(n_1+m_2) \delta_{i,j}\delta_{m_1,m_2}\\
= &  e_{i,j}(m_1+n_1) - \delta_{i,j}  e_{1,1}(m_1+n_1),
\end{align*}
for $i,\,j=2,\ldots,m$.
\begin{align*}
\big[\ e_{m+i,1}(m_1),\ e_{1,j}(n_1) \big]
= & \summpnp  \big[ e_{m+i,1}(m_1), \ e_{1,1}(n_1+n_2) \big] \p{x_j(n_2)}\\
= & \summpnp e_{m+i,1}(m_1+n_1+n_2)\p{x_j(n_2)}\\
= & e_{m+i,j}(m_1+n_1),
      \end{align*}
for $i=1,\ldots,n$ and $j=2,\ldots,m.$
\begin{align*}
\big[ e_{i,j}(m_1), \ e_{1,k}(n_1) \big]
= & \summnmpnp  e_{1,1}(n_1+n_2) \bigg[ x_i(m_1+m_2) \p{x_j(m_2)}, \p{x_k(n_2)} \bigg] \\
= & -\summnmpnp  e_{1,1}(n_1+n_2)\bigg[ \p{x_k(n_2)}, x_i(m_1 + m_2) \bigg]  \p{x_j(m_2)}  \\
= & -\delta_{ik}  e_{1,j}(m_1+n_1),
\end{align*}
for $i,\,j,\,k=2,\ldots,m$.
\begin{align*}
\big[ e_{i,m+j}(m_1), \ e_{1,k}(n_1) \big]
= & \summnmpnp  e_{1,1}(n_1+n_2) \bigg[ x_i(m_1+m_2) \p{y_j(m_2)}, \p{x_k(n_2)} \bigg] \\
= & -\summnmpnp  e_{1,1}(n_1+n_2)\bigg[ \p{x_k(n_2)}, x_i(m_1 + m_2) \bigg] \p{y_j(m_2)}\\
= & -\delta_{ik}  e_{1,m+j}(m_1+n_1),
\end{align*}
for $i,\,k=2,\ldots,m$ and $j=1,\ldots,n$.
\begin{align*}
\big[ e_{m+i,j}(m_1), \ e_{1,k}(n_1) \big]
= \summnmpnp  e_{1,1}(n_1+n_2) \big[ y_i(m_1+m_2) \p{x_j(m_2)}, \p{x_k(n_2)} \big]
= 0,
\end{align*}
for $i=1,\ldots,n$ and $j,\,k=2,\ldots,m$.
\begin{align*}
\big[ e_{m+i,m+j}(m_1), \ e_{1,k}(n_1) \big]
= \summnmpnp  e_{1,1}(n_1+n_2) \big[ y_i(m_1+m_2) \p{y_j(m_2)}, \p{x_k(n_2)} \big]
= 0,
\end{align*}
for $i,\,j=1,\ldots,n$ and $k=2,\ldots,m$.

\begin{align*}
\big\{ e_{1,m+i}(m_1),& \ e_{1,m+j}(n_1) \big\}_+ \\
= & \summn  \big[ e_{1,m+i}(m_1), \ e_{1,1}(n_1+m_2) \big] \p{y_j(m_2)} \\
& + \summn  e_{1,1}(n_1+m_2) \bigg\{ \summpnp  e_{1,1}(m_1+n_2) \p{y_i(n_2)}, \p{y_j(m_2)} \bigg\}_+ \\
= & -  \summn  e_{1,m+i}(m_1+n_1+m_2) \p{y_j(m_2)}+ \summnmpnp  e_{1,1}(n_1+m_2) \Bigg\{ \bigg( \mu \delta_{m_1+n_2,0} \\
&  - \summppnpp \sum_{k=1}^{n}  y_k(m_1+n_2+n_3) \p{y_k(n_3)} \bigg) \p{y_i(n_2)}, \p{y_j(m_2)} \Bigg\}_+ \\
= & -  \summn  e_{1,m+i}(m_1+n_1+m_2) \p{y_j(m_2)} \\
& - \summnmpnpmppnpp \sum_{k=1}^{n}  e_{1,1}(n_1+m_2) \bigg\{ \p{y_j(m_2)}, y_k(m_1+n_2+n_3) \bigg\}_+\p{y_k(n_3)} \p{y_i(n_2)} \\
= & -  \summn  e_{1,m+i}(m_1+n_1+m_2) \p{y_j(m_2)}- \summpnpmppnpp  e_{1,1}(n_1+m_1+n_2+n_3) \p{y_j(n_3)} \p{y_i(n_2)} \\
= &  0,
\end{align*}
for $i,j=1,\ldots,n$.	
\begin{align*}
\big[\ e_{i,1}(m_1),\ \ e_{1,m+j}(n_1)\big]
= & \summpnp  \big[ e_{i,1}(m_1), \ e_{1,1}(n_1+n_2) \big] \p{y_j(n_2)}\\
= & \summpnp e_{i,1}(m_1+n_1+n_2)\p{y_j(n_2)}\\
= & e_{i,m+j}(m_1+n_1),
\end{align*}
for $i=2,\ldots,m$ and $j=1,\ldots,n.$
\begin{align*}
\big\{ e_{m+i,1}(m_1), & \ e_{1,m+j}(n_1) \big\}_+ \\
= & \ \bigg\{ e_{m+i,1}(m_1), \ \sum_{m_2\in\bbZ}  e_{1,1}(n_1+m_2) \p{y_j(m_2)} \bigg\}_+ \\
= & \sum_{m_2\in\bbZ}  \Bigg( \bigg[ e_{m+i,1}(m_1),\  e_{1,1}(n_1+m_2) \bigg] \p{y_j(m_2)}
      + e_{1,1}(n_1+m_2) \bigg\{ y_i(m_1), \p{y_j(m_2)} \bigg\}_+ \Bigg) \\
= &  e_{m+i,m+j}(m_1+n_1) + \sum_{m_2\in\bbZ}  e_{1,1}(n_1+m_2) \delta_{i,j}\delta_{m_1,m_2}\\
= &  e_{m+i,m+j}(m_1+n_1) + \delta_{i,j}  e_{1,1}(m_1+n_1),
\end{align*}
for $i,\,j=1,\ldots,n$.
\begin{align*}
\big[ e_{i,j}(m_1),&  \ e_{1,m+k}(n_1) \big] \\
= & \summpnp  \big[ e_{i,j}(m_1), \ e_{1,1}(n_1+n_2) \big] \p{y_k(n_2)}
  + \summnmpnp  e_{1,1}(n_1+n_2) \bigg[ x_i(m_1+m_2) \p{x_j(m_2)}, \p{y_k(n_2)} \bigg] \\
= & 0,
\end{align*}
for $i,\,j=2,\ldots,m$ and $k=1,\ldots,n$.
\begin{align*}
\big\{&e_{i, m+j}(m_1), \ e_{1,m+k}(n_1) \big\}_+ \\
& = \summpnp  \big[ e_{i,m+j}(m_1), \ e_{1,1}(n_1+n_2) \big] \p{y_k(n_2)}
    + \summnmpnp  e_{1,1}(n_1+n_2) \bigg\{ x_i(m_1+m_2) \p{y_j(m_2)}, \p{y_k(n_2)} \bigg\}_+ \\
& = 0,
\end{align*}
for $i=2,\ldots,m$ and $j,\,k=1,\ldots,n$.
\begin{align*}
\big\{& e_{m+i,j}(m_1), \ e_{1,m+k}(n_1) \big\}_+ \\
& = \summpnp  \big[ e_{m+i,j}(m_1), \ e_{1,1}(n_1+n_2) \big] \p{y_k(n_2)}
     + \summnmpnp  e_{1,1}(n_1+n_2) \bigg\{ y_i(m_1+m_2) \p{x_j(m_2)}, \p{y_k(n_2)} \bigg\}_+ \\
& = \summnmpnp  e_{1,1}(n_1+n_2)\bigg\{ \p{y_k(n_2)}, y_i(m_1 + m_2) \bigg\}_+  \p{x_j(m_2)}  \\
& = \delta_{i,k}e_{1,j}(m_1+n_1),
\end{align*}
for $i,\,k=1,\ldots,n$ and $j=2,\ldots,m$.
\begin{align*}
\big[ & e_{m+i,m+j}(m_1), \ e_{1,m+k}(n_1) \big] \\
& = \summpnp  \big[ e_{m+i,j}(m_1), \ e_{1,1}(n_1+n_2) \big] \p{y_k(n_2)}
     + \summnmpnp  e_{1,1}(n_1+n_2) \bigg[ y_i(m_1+m_2) \p{y_j(m_2)}, \p{y_k(n_2)} \bigg] \\
& = -\summnmpnp  e_{1,1}(n_1+n_2) \bigg\{ \p{y_k(n_2)}, y_i(m_1 + m_2) \bigg\}_+  \p{y_j(m_2)} \\
& = -\delta_{i,k}e_{1,m+j}(m_1+n_1),
\end{align*}
for $i,\,j,\,k=1,\ldots,n$.
\begin{displaymath}
\big[ e_{i,1}(m_1), \ e_{j,1}(n_1) \big] \ = \ \big[ x_i(m_1), \ x_j(n_1) \big] \ = \ 0,
\end{displaymath}
for $i,\,j=2,\ldots,m.$
\begin{displaymath}
\big[ e_{i,1}(m_1), \ e_{m+j,1}(n_1) \big] \ = \ \big[ x_i(m_1), \ y_j(n_1) \big] \ = \ 0,
\end{displaymath}
for $i=2,\ldots,m$ and $j=1,\ldots,n$.
\begin{align*}
\big[ e_{i,1}(m_1), \ e_{j,k}(n_1) \big] \ = &\bigg[ x_i(m_1),\ \sum_{m_2 \in \bbZ} x_j(n_1+m_2) \p{x_k(m_2)} \bigg] \\
= & \sum_{m_2\in\bbZ}  x_j(n_1+m_2)  \bigg[ x_i(m_1), \p{x_k(m_2)}\bigg]  \\
= & -\sum_{m_2\in\bbZ}  \delta_{i,k} \delta_{m_1, m_2} x_j(n_1+m_2) \\
= & - \delta_{i,k}  e_{j,1}(m_1+n_1),
	\end{align*}
for $i,\,j,\,k=2,\ldots,m$.
\begin{align*}
\big[ e_{i,1}(m_1), \ e_{j,m+k}(n_1) \big] \ = \ \bigg[ x_i(m_1),\ \sum_{m_2 \in \bbZ} x_j(n_1+m_2) \p{y_k(m_2)} \bigg] = 0,
\end{align*}
for $i,\,j=2,\ldots,m$ and $k=1,\dots,n.$
\begin{align*}
\big[ e_{i,1}(m_1), \ e_{m+j,k}(n_1) \big] \ = & \bigg[ x_i(m_1),\ \sum_{m_2 \in \bbZ} y_j(n_1+m_2)
\p{x_k(m_2)} \bigg] \\
= & -\sum_{m_2\in\bbZ}   y_j(n_1+m_2)  \bigg[ \p{x_k(m_2)}, x_i(m_1)\bigg] \\
= & -\sum_{m_2\in\bbZ}  \delta_{i,k} \delta_{m_1, m_2} y_j(n_1+m_2) \\
= & - \delta_{i,k}  e_{m+j,1}(m_1+n_1),
\end{align*}
for $i,\,k=2,\ldots,m$ and $j=1,\dots,n.$
\begin{align*}
\big[ e_{i,1}(m_1), \ e_{m+j,m+k}(n_1) \big] \ = \bigg[ x_i(m_1),\ \sum_{m_2 \in \bbZ} y_j(n_1+m_2) \p{y_k(m_2)} \bigg] =  0,
\end{align*}
for $i=2,\ldots,m$ and $j,\,k=1,\dots,n.$
\begin{displaymath}
\big\{ e_{m+i,1}(m_1), \ e_{m+j,1}(n_1) \big\}_+ \ = \ \big\{ y_i(m_1), \ y_j(n_1) \big\}_+ \ = \ 0,
\end{displaymath}
for $i,\,j=1,\ldots,n $.
\begin{align*}
\big[ e_{m+i,1}(m_1), \ e_{j,k}(n_1) \big] \ = \bigg[ y_i(m_1),\ \sum_{m_2 \in \bbZ} x_j(n_1+m_2) \p{x_k(m_2)} \bigg]=  0,
\end{align*}
for $i=1,\ldots,n$ and $j,\,k=2,\ldots,m$.
\begin{align*}
\big\{ e_{m+i,1}(m_1), \ e_{j,m+k}(n_1) \big\}_+ \
= & \bigg\{ y_i(m_1),\ \sum_{m_2 \in \bbZ}  x_j(n_1+m_2) \p{y_k(m_2)} \bigg\}_+ \\
= & \sum_{m_2\in\bbZ}  x_j(n_1+m_2)  \bigg\{ y_i(m_1), \p{y_k(m_2)}\bigg\}_+  \\
= & \sum_{m_2\in\bbZ}\delta_{i,k}\delta_{m_1,m_2}x_j(n_1+m_2)\\
= & \delta_{i,k}e_{j,1}(m_1+n_1),
\end{align*}
for $i,\,k=1,\ldots,n$ and $j=2,\ldots,m.$
\begin{align*}
\big\{ e_{m+i,1}(m_1), \ e_{m+j,k}(n_1) \big\}_+ \ = \ \bigg\{ y_i(m_1),\ \sum_{m_2 \in \bbZ}  y_j(n_1+m_2) \p{x_k(m_2)} \bigg\}_+= 0,
\end{align*}
for $i,\,j=1,\ldots,n$ and $k=2,\ldots,m$.
\begin{align*}
\big[ e_{m+i,1}(m_1), \ e_{m+j,m+k}(n_1) \big] \
= &\bigg[ y_i(m_1),\ \sum_{m_2 \in \bbZ}  y_j(n_1+m_2) \p{y_k(m_2)} \bigg]\\
= & -\sum_{m_2\in\bbZ}  y_j(n_1+m_2)  \bigg\{ y_i(m_1), \p{y_k(m_2)}\bigg\}_+   \\
= & -\sum_{m_2\in\bbZ}   \delta_{i,k} \delta_{m_1,m_2}  y_j(n_1+m_2) \\
= &  -\delta_{i,k}  e_{j,1}(m_1+n_1),
\end{align*}
for $i,\,j,\,k=1,\ldots,n$.
\begin{align*}
\big[ e_{i,j}(m_1), \ e_{k,l}(n_1) \big]
= &  \summnmpnp  \bigg[ x_i(m_1+m_2) \p{x_j(m_2)}, x_k(n_1+n_2) \p{x_l(n_2)} \bigg] \\
= &  \summnmpnp \Bigg( x_i(m_1+m_2) \bigg[ \p{x_j(m_2)}, x_k(n_1+n_2) \bigg] \p{x_l(n_2)} \\
  & - x_k(n_1+n_2) \bigg[ \p{x_l(n_2)}, x_i(m_1+m_2) \bigg]  \p{x_j(m_2)} \Bigg) \\
= & \delta_{j,k}  e_{i,l}(m_1+n_1) - \delta_{l,i}  e_{k,j}(m_1+n_1),
\end{align*}
for $i,\,j,\,k,\,l=2,\ldots,m.$
\begin{align*}
\big[ e_{i,j}(m_1), \ e_{k,m+l}(n_1) \big]
= &  \summnmpnp  \bigg[ x_i(m_1+m_2) \p{x_j(m_2)}, x_k(n_1+n_2) \p{y_l(n_2)} \bigg] \\
= &  \summnmpnp  x_i(m_1+m_2)  \bigg[ \p{x_j(m_2)}, x_k(n_1+n_2) \bigg] \p{y_l(n_2)}\\
= & \delta_{j,k}e_{i,m+r}(m_1+n_1),
\end{align*}
for $i,\,j,\,k=2,\ldots,m$ and $l=1,\ldots,n.$
\begin{align*}
\big[ e_{i,j}(m_1), \ e_{m+k,l}(n_1) \big]
= &  \summnmpnp  \bigg[ x_i(m_1+m_2) \p{x_j(m_2)}, y_k(n_1+n_2) \p{x_l(n_2)} \bigg] \\
= &  -\summnmpnp y_k(n_1+n_2) \bigg[ \p{x_l(n_2)}, x_i(m_1+m_2) \bigg]  \p{x_j(m_2)} \\
= & -\delta_{i,l}e_{m+k,j}(m_1+n_1),
	\end{align*}
for $i,\,j,\,l=2,\ldots,m$ and $k=1,\ldots,n.$
\begin{align*}
\big[ e_{i,j}(m_1), \ e_{m+k,m+l}(n_1) \big]=\summnmpnp\bigg[ x_i(m_1+m_2) \p{x_j(m_2)}, y_k(n_1+n_2) \p{y_l(n_2)} \bigg] =  0,
\end{align*}
for $i,\,j=2,\ldots,m$ and $k,\,l=1,\ldots,n.$

\begin{align*}
\big\{ e_{i,m+j}(m_1), \ e_{k,m+l}(n_1) \big\}_+ = &  \summnmpnp  \bigg\{ x_i(m_1+m_2) \p{y_j(m_2)}, x_k(n_1+n_2) \p{y_l(n_2)} \bigg\}= 0,
\end{align*}
for $i,\,k=2,\ldots,m$ and $j,\,l=1,\ldots,n$.	
\begin{align*}
\big\{ e_{i,m+j}(m_1), \ e_{m+k,l}(n_1) \big\}_+
= &  \summnmpnp  \bigg\{ x_i(m_1+m_2) \p{y_j(m_2)}, y_k(n_1+n_2) \p{x_l(n_2)} \bigg\} \\
= &  \summnmpnp  \Bigg( x_i(m_1+m_2) \bigg\{ \p{y_j(m_2)}, y_k(n_1+n_2) \bigg\}_+ \p{x_l(n_2)} \\
  & + y_k(n_1+n_2) \bigg[ \p{x_l(n_2)}, x_i(m_1+m_2) \bigg] \p{y_j(m_2)} \Bigg) \\
= & \delta_{j,k} e_{i,l}(m_1+n_1) + \delta_{l,i}  e_{m+k,m+j}(m_1+n_1),
\end{align*}
for $i,\,r=l,\ldots,m$ and $k,\,j=1,\ldots,n$.
\begin{align*}
\big[ e_{i,m+j}(m_1), \ e_{m+k,m+l}(n_1) \big]
= &  \summnmpnp  \bigg[ x_i(m_1+m_2) \p{y_j(m_2)}, y_k(n_1+n_2) \p{y_l(n_2)} \bigg] \\
= &  \summnmpnp   x_i(m_1+m_2) \bigg\{ \p{y_j(m_2)}, y_k(n_1+n_2) \bigg\}_+ \p{y_l(n_2)} \\
= & \delta_{j,k} e_{i,m+l}(m_1+n_1),
\end{align*}
for $i,=2,\ldots,m$ and $j,\,k,\,l=1,\ldots,n$.
\begin{align*}
\big\{ e_{m+i,j}(m_1), \ e_{m+k,l}(n_1) \big\}_+
= \summnmpnp\bigg\{ y_i(m_1+m_2) \p{x_j(m_2)}, y_k(n_1+n_2) \p{x_l(n_2)} \bigg\}
= 0,
\end{align*}
for $i,\,k=1,\ldots,n$ and $j,\,l=2,\ldots,m$.	
\begin{align*}
\big[ e_{m+i,j}(m_1), \ e_{m+k,m+l}(n_1) \big]
= &  \summnmpnp  \bigg[ y_i(m_1+m_2) \p{x_j(m_2)}, y_k(n_1+n_2) \p{y_l(n_2)} \bigg] \\
= &  -\summnmpnp y_k(n_1+n_2) \bigg\{ \p{y_l(n_2)}, y_i(m_1+m_2) \bigg\}_+ \p{x_j(m_2)} \\
= & -\delta_{i,l}e_{m+k,j}(m_1+n_1),
\end{align*}
for $i,\,k,\,l=1,\ldots,n$ and $j=2,\ldots,m$.	
\begin{align*}
\big[ e_{m+i,m+j}(m_1), \ e_{m+k,m+l}(n_1) \big]
= &  \summnmpnp  \bigg[ y_i(m_1+m_2) \p{y_j(m_2)}, y_k(n_1+n_2) \p{y_l(n_2)} \bigg] \\
= &  \summnmpnp \Bigg( y_i(m_1+m_2)\bigg\{ \p{y_j(m_2)}, y_k(n_1+n_2) \bigg\}_+ \p{y_l(n_2)} \\
& - y_k(n_1+n_2) \bigg\{ \p{y_l(n_2)}, y_i(m_1+m_2) \bigg\}_+ \p{y_j(m_2)} \Bigg) \\
= & \delta_{j,k}  e_{m+i,m+l}(m_1+n_1) - \delta_{l,i}  e_{m+k,m+j}(m_1+n_1),
\end{align*}
for $i,\,j,\,k,\,l=1,\ldots,n.$
\end{proof}

We now define the following linear operator on $\widehat{W}$:
\begin{equation}
\begin{split}
D &= \sum_{i=2}^{m}\sum_{m_1\in \bbZ}m_1 x_i(m_1)\p{x_i(m_1)}+\sum_{j=1}^{n}\sum_{m_1\in \bbZ}m_1 y_j(m_1)\p{y_j(m_1)} \\
\end{split}
\end{equation}

\begin{cor}  \label{bigrep}
There is a Lie superalgebra homomorphism
\begin{displaymath}
\psi: \widehat{\mathfrak{gl}_{m|n}}(\bbC) \longrightarrow \mathfrak{gl}(\widehat{W})
\end{displaymath}
given by
\begin{equation*}
\psi\big( E_{i,j} \otimes t^{m_1} \big) \ = \ e_{i,j}(m_1), \quad \psi(K)=0, \quad \psi(d) = D
\end{equation*}
where $i,j =1, \ldots, m+n$.
Thus, $\widehat{W}$ is also a module over the Lie superalgebra $\widehat{\mathfrak{gl}_{m|n}}(\bbC)$.
\end{cor}

\begin{proof}
By Theorem \ref{mainrep1}, it suffices to show
\begin{equation*}
[D, e_{i,j}(m_1)]=m_1 e_{i,j}(m_1).
\end{equation*}

First we have
\begin{equation*}
[D,ab] = [D,a]b +a[D,b]
\end{equation*}
for elements $D,a,b \in \mathfrak{gl}(\widehat{W})$ with $D$ even.
Now we have the following relations
\begin{equation*}
\begin{split}
\big[D, x_j(m_1)\big]
&=  \sum_{i=2}^{m}\sum_{m_2\in \bbZ}m_2 \bigg[x_i(m_2)\p{x_i(m_2)}, x_j(m_1)\bigg]+\sum_{i=1}^{n}\sum_{m_2\in \bbZ}m_2 \bigg[y_i(m_2)\p{y_i(m_2)}, x_j(m_1)\bigg]  \\
&=  \sum_{i=2}^{m}\sum_{m_2\in \bbZ}m_2 x_i(m_2)\bigg[\p{x_i(m_2)}, x_j(m_1)\bigg]  \\
&=  \sum_{i=2}^{m}\sum_{m_2\in \bbZ}m_2 x_i(m_2) \delta_{i,j}\delta_{m_1,m_2}  \\
&=  m_1 x_j(m_1)
\end{split}
\end{equation*}
for $j=2,\ldots,m$.
\begin{equation*}
\begin{split}
\big[D, y_j(m_1)\big]
&=  \sum_{i=2}^{m}\sum_{m_2\in \bbZ}m_2 \bigg[x_i(m_2)\p{x_i(m_2)}, y_j(m_1)\bigg]+\sum_{i=1}^{n}\sum_{m_2\in \bbZ}m_2 \bigg[y_i(m_2)\p{y_i(m_2)}, y_j(m_1)\bigg]  \\
&=  \sum_{i=1}^{n}\sum_{m_2\in \bbZ}m_2 y_i(m_2)\bigg\{\p{y_i(m_2)}, y_j(m_1)\bigg\}_+  \\
&=  \sum_{i=1}^{n}\sum_{m_2\in \bbZ}m_2 y_i(m_2) \delta_{i,j}\delta_{m_1,m_2}  \\
&=  m_1 y_j(m_1)
\end{split}
\end{equation*}
for $j=1,\ldots,n$.
\begin{equation*}
\begin{split}
\bigg[D, \p{x_j(m_1)}\bigg]
&= \sum_{i=2}^{m}\sum_{m_2\in \bbZ}m_2 \bigg[x_i(m_2)\p{x_i(m_2)}, \p{x_j(m_1)}\bigg]+\sum_{i=1}^{n}\sum_{m_2\in \bbZ}m_2 \bigg[y_i(m_2)\p{y_i(m_2)}, \p{x_j(m_1)}\bigg]  \\
&=  -\sum_{i=2}^{m}\sum_{m_2\in \bbZ}m_2 \bigg[\p{x_j(m_1)}, x_i(m_2)\bigg]\p{x_i(m_2)}  \\
&=  -\sum_{i=2}^{m}\sum_{m_2\in \bbZ}m_2 \p{x_i(m_2)} \delta_{i,j}\delta_{m_1,m_2}  \\
&=  -m_1 \p{x_j(m_1)}
\end{split}
\end{equation*}
for $j=2,\ldots,m$.
\begin{equation*}
\begin{split}
\bigg[D, \p{y_j(m_1)}\bigg]
&=  \sum_{i=2}^{m}\sum_{m_2\in \bbZ}m_2 \bigg[x_i(m_2)\p{x_i(m_2)}, \p{y_j(m_1)}\bigg]+\sum_{i=1}^{n}\sum_{m_2\in \bbZ}m_2 \bigg[y_i(m_2)\p{y_i(m_2)}, \p{y_j(m_1)}\bigg]  \\
&=  -\sum_{i=1}^{n}\sum_{m_2\in \bbZ}m_2 \bigg\{\p{y_j(m_1)}, y_i(m_2)\bigg\}_+\p{y_i(m_2)}  \\
&=  -\sum_{i=1}^{n}\sum_{m_2\in \bbZ}m_2  \delta_{i,j}\delta_{m_1,m_2}\p{y_i(m_2)}  \\
&=  -m_1 \p{y_j(m_1)}
\end{split}
\end{equation*}
for $j=1,\ldots,n$.

Whence, for $i,\,j=2,\ldots,m$ we get
\begin{align*}
[D,e_{i,j}(m_1)]
&=\sum\limits_{m_2\in\bbZ}\Bigg([D,\ x_i(m_1+m_2)]\p{x_j(m_2)}+x_i(m_1+m_2)\bigg[D,\ \p{x_j(m_2)}\bigg]\Bigg)\\
&=(m_1+m_2)e_{i,j}(m_1)-m_2e_{i,j}(m_1)\\
&=m_1e_{i,j}(m_1)
\end{align*}
Similarly, we could show that
$$
[D, e_{i,j}(m_1)]=m_1e_{i,j}(m_1)
$$
for $i,\,j=1,\ldots,m+n$. Now the proof is completed.
\end{proof}

In the following, for convenience we denote $x_{m+k}(m_1)=y_k(m_1)$ for $k=1,\ldots,n$ and $m_1\in\bbZ$.
It is clear that $\widehat{W}=U(\widehat{\mathfrak{gl}_{m|n}}(\bbC)).1$, and
\begin{align*}
& e_{1,1} (0) \Big( x_2(m_{2,1})\ldots x_2(m_{2,k_2}) \ldots x_{m+n}(m_{m+n,1})\ldots x_{m+n}(m_{m+n,k_{m+n}})\Big) \\
= & (\mu-k_2-\ldots-k_{m+n})x_2(m_{2,1})\ldots x_2(m_{2,k_2})\ldots x_{m+n}(m_{m+n,1})\ldots x_{m+n}(m_{m+n,k_{m+n}})
\end{align*}
and
\begin{align*}
&e_{i,i}(0)\Big(x_2(m_{2,1})\ldots x_2(m_{2,k_2})\ldots x_{m+n}(m_{m+n,1})\ldots x_{m+n}(m_{m+n,k_{m+n}})\Big)\\
= & k_ix_2(m_{2,1})\ldots x_2(m_{2,k_2})\ldots x_{m+n}(m_{m+n,1})\ldots x_{m+n}(m_{m+n,k_{m+n}}).
\end{align*}
for $i=2\ldots,m+n$. For the derivation $D$, we have
\begin{align*}
& D\Big(x_2(m_{2,1})\ldots x_2(m_{2,k_2})\ldots x_{m+n}(m_{m+n,1})\ldots x_{m+n}(m_{m+n,k_{m+n}})\Big)\\
= & \Big(\sum\limits_{i=2}^{m+n}\sum\limits_{j=1}^{k_{i}}m_{i,j}\Big)x_2(m_{2,1})\ldots x_2(m_{2,k_2})\ldots x_{m+n}(m_{m+n,1})\ldots x_{m+n}(m_{m+n,k_{m+n}})
\end{align*}
In the following, the weight $ \lambda\in H^* $ will be denoted as
$$
\Big(\lambda\big(E_{1,1}(0)\big),\ldots,\lambda\big(E_{m+n,m+n}(0)\big),0,\lambda(d)\Big).
$$

\begin{thm}
If $\mu\ne 0$, then $\widehat{W}$ is an infinite-dimensional irreducible $\widehat{\mathfrak{gl}_{m|n}}(\bbC)$-module.
\end{thm}

\begin{proof}
We split the proof into three cases: $\widehat{\mathfrak{gl}_{m|0}}(\bbC)=\widehat{\mathfrak{gl}_m}(\bbC)$,
$\widehat{\mathfrak{gl}_{1|n}}(\bbC)$ and $\widehat{\mathfrak{gl}_{m|n}}(\bbC)$ ($m\geq2$).
The first case was proved by Gao and Zeng in \cite{GZ2}.
For second case, we have
\begin{align*}
e_{1,j}(r&) \Big(x_2(m_{2,1}) \cdots x_2(m_{2,k_2}) x_3(m_{3,1}) \cdots x_3(m_{3,k_3}) \cdots x_{n+1}(m_{n+1,1}) \cdots
    x_{n+1}(m_{n+1,k_{n+1}}) \Big)     \\
= & \sum_{i=1}^{k_2} (-1)^{i-1} e_{2,j}(r+m_{2,i})e_{2,1}(m_{2,1}) \cdots \widehat{e_{2,1}(m_{2,i})} \cdots e_{2,1}(m_{2,k_2})
    \cdots e_{n+1,1}(m_{n+1,1}) \cdots e_{n,1}(m_{n+1,k_{n+1}}).1   \\
& +(-1)^{k_2} e_{2,1}(m_{2,1}) \cdots e_{2,1}(m_{2,k_2})e_{1,j}(r) e_{3,1}(m_{3,1}) \cdots e_{3,1}(m_{3,k_3}) \cdots e_{n+1,1}(m_{n+1,k_{n+1}}).1     \\
= & \cdots \cdots  \\
= & \sum_{l=2}^{j-1}\sum_{i=1}^{k_l} (-1)^{\sum\limits_{s=2}^{l-1}k_s+i-1} e_{l,j}(r+m_{l,i})
    e_{2,1}(m_{2,1}) \cdots \widehat{e_{l,1}(m_{l,i})} \cdots e_{n+1,1}(m_{n+1,1}) \cdots e_{n+1,1}(m_{n+1,k_{n+1}}).1   \\
  & +(-1)^{\sum\limits_{i=2}^{j-1}k_i}e_{2,1}(m_{2,1}) \cdots e_{2,1}(m_{2,k_2}) \cdots e_{1,j}(r) e_{j,1}(m_{j,1})
    \cdots e_{j,1}(m_{j,k_j}) \cdots e_{n+1,1}(m_{n+1,k_{n+1}}).1  \\
= & \sum_{l=2}^{j-1}\sum_{i=1}^{k_l} (-1)^{\sum\limits_{s=2}^{l-1}k_s+i-1} \bigg(e_{2,1}(m_{2,1}) \cdots
    \widehat{e_{l,1}(m_{l,i})} \cdots e_{n+1,1}(m_{n+1,1}) \cdots e_{n+1,1}(m_{n+1,k_{n+1}})e_{l,j}(r+m_{l,i}).1   \\
  & +\sum_{k=1}^{k_j} e_{2,1}(m_{2,1}) \cdots \widehat{e_{l,1}(m_{l,i})} \cdots
     \widehat{e_{j,1}(m_{j,k})} e_{l,1}(r+m_{l,i}+m_{j,k})\cdots e_{n+1,1}(m_{n+1,k_{n+1}}).1\bigg)   \\
  & +(-1)^{\sum\limits_{i=2}^{j-1}k_i}\sum_{k=1}^{k_j}(-1)^{k-1}
     e_{2,1}(m_{2,1}) \cdots \widehat{e_{j,1}(m_{j,k})} (e_{1,1}+ e_{j,j})(r+m_{j,k}) \cdots e_{n+1,1}(m_{n+1,k_{n+1}}).1 \\
  & +(-1)^{\sum\limits_{i=2}^{j}k_i}e_{2,1}(m_{2,1}) \cdot \cdot e_{j,1}(m_{j,1})
    \cdots e_{j,1}(m_{j,k_j}) e_{1,j}(r)e_{j+1,1}(m_{j+1,1}) \cdots e_{n+1,1}(m_{n+1,k_{n+1}}).1  \\
= & \sum_{l=2}^{j-1}\sum_{i=1}^{k_l}\sum_{k=1}^{k_j} (-1)^{\sum\limits_{s=2}^{l-1}k_s+i-1} e_{2,1}(m_{2,1}) \cdots
    \widehat{e_{l,1}(m_{l,i})} \cdots \widehat{e_{j,1}(m_{j,k})} e_{l,1}(r+m_{l,i}+m_{j,k}) \cdots .1   \\
  & -(-1)^{\sum\limits_{i=2}^{j-1}k_i}\sum_{k=1}^{k_j}(-1)^{k} e_{2,1}(m_{2,1}) \cdots
    \widehat{e_{j,1}(m_{j,k})}\cdots (e_{1,1}+ e_{j,j})(r+m_{j,k}) e_{j+1,1}(m_{j+1,1}) \cdots .1 \\
  & +\sum_{l=j+1}^{n+1} \sum_{i=1}^{k_l} (-1)^{\sum\limits_{s=2}^{l-1}k_s+i-1}
    e_{2,1}(m_{2,1})  \cdots \widehat{e_{l,1}(m_{l,i})} \cdots \cdots e_{n+1,1}(m_{n+1,k_{n+1}})e_{l,j}(r+m_{l,i}).1  \\
  & + (-1)^{\sum\limits_{i=2}^{n+1}k_i} e_{2,1}(m_{2,1}) \cdots e_{2,1}(m_{2,k_2})
    \cdots \cdots e_{n+1,1}(m_{n+1,1}) \cdots e_{n+1,1}(m_{n+1,k_{n+1}}) e_{1,j}(r).1  \\
= & \sum_{l=2}^{j-1}\sum_{i=1}^{k_l}\sum_{k=1}^{k_j}
    (-1)^{\sum\limits_{s=2}^{l-1}k_s+i-1} e_{2,1}(m_{2,1}) \cdots
    \widehat{e_{l,1}(m_{l,i})} \cdots \widehat{e_{j,1}(m_{j,k})}
    e_{l,1}(r+m_{l,i}+m_{j,k}) \cdots .1   \\
  & +(-1)^{\sum\limits_{i=2}^{j-1}k_i} \sum_{l=j+1}^{n}\sum_{k=1}^{k_j}\sum_{i=1}^{k_l}(-1)^{k} e_{2,1}(m_{2,1}) \cdots
    \widehat{e_{j,1}(m_{j,k})} \cdots \widehat{e_{l,1}(m_{l,i})}e_{l,1}(r+m_{l,i}+m_{j,k}) \cdots .1   \\
  & -(-1)^{\sum\limits_{i=2}^{j-1}k_i}\sum_{k=1}^{k_j}(-1)^{k} e_{2,1}(m_{2,1}) \cdots \widehat{e_{j,1}(m_{j,k})}\cdots
    e_{n+1,1}(m_{n+1,k_{n+1}})(e_{1,1}+ e_{j,j})(r+m_{j,k}).1  \\
= & \sum_{l=2}^{j-1}\sum_{i=1}^{k_l}\sum_{k=1}^{k_j} (-1)^{\sum\limits_{s=2}^{l-1}k_s+i-1} e_{2,1}(m_{2,1}) \cdots
    \widehat{e_{l,1}(m_{l,i})} \cdots \widehat{e_{j,1}(m_{j,k})} e_{l,1}(r+m_{l,i}+m_{j,k}) \cdots .1   \\
  & +(-1)^{\sum\limits_{i=2}^{j-1}k_i} \sum_{l=j+1}^{n+1}\sum_{k=1}^{k_j}\sum_{i=1}^{k_l}(-1)^{k} e_{2,1}(m_{2,1}) \cdots
    \widehat{e_{j,1}(m_{j,k})} \cdots  \widehat{e_{l,1}(m_{l,i})}e_{l,1}(r+m_{l,i}+m_{j,k}) \cdots .1   \\
  & -(-1)^{\sum\limits_{i=2}^{j-1}k_i} \mu \sum_{k=1}^{k_j}\delta_{r+m_{j,k},0} (-1)^{k}
    e_{2,1}(m_{2,1}) \cdots \widehat{e_{j,1}(m_{j,k})}\cdots e_{n+1,1}(m_{n+1,1})\cdots e_{n+1,1}(m_{n+1,k_{n+1}}).1 \\
= & A_j(r)+B_j(r)
\end{align*}
where
\begin{align*}
A_j(r) = & -(-1)^{\sum\limits_{i=2}^{j-1}k_i} \mu \sum_{k=1}^{k_j}\delta_{r+m_{j,k},0} (-1)^{k}
            e_{2,1}(m_{2,1}) \cdots \widehat{e_{j,1}(m_{j,k})}\cdots e_{n+1,1}(m_{n+1,1})\cdots e_{n+1,1}(m_{n+1,k_{n+1}}).1 \, , \\
B_j(r) = & \sum_{l=2}^{j-1}\sum_{i=1}^{k_l}\sum_{k=1}^{k_j} (-1)^{\sum\limits_{s=2}^{l-1}k_s+i-1} e_{2,1}(m_{2,1}) \cdots
       \widehat{e_{l,1}(m_{l,i})} \cdots \widehat{e_{j,1}(m_{j,k})}  e_{l,1}(r+m_{l,i}+m_{j,k}) \cdots .1   \\
  & +(-1)^{\sum\limits_{i=2}^{j-1}k_i} \sum_{l=j+1}^{n+1}\sum_{k=1}^{k_j}\sum_{i=1}^{k_l}(-1)^{k}
    e_{2,1}(m_{2,1}) \cdots \widehat{e_{j,1}(m_{j,k})} \cdots \widehat{e_{l,1}(m_{l,i})}e_{l,1}(r+m_{l,i}+m_{j,k}) \cdots . 1.
\end{align*}
Let $Y=Y_0+\sum\limits_{\alpha \in I}Y_{\alpha} e_{2,1}(m^{\alpha}_{2,1}) \cdots e_{2,1}(m^{\alpha}_{2,k_{2,\alpha}})
\cdots e_{n+1,1}(m^{\alpha}_{n+1,1}) \cdots e_{n+1,1}(m^{\alpha}_{n+1,k_{n+1,\alpha}})$
be a singular vector of a submodule $U$ of $\widehat{W}$,
where $k_{\alpha}=\sum\limits_{i=2}^{n+1}k_{i,\alpha} \ge 1$,
$Y_0 \in \bbC$ and $m^{\alpha}_{l,1}< \cdots < m^{\alpha}_{l,k_{\alpha}}$,
then $e_{1,j}(r)Y=0$, for all $2 \le j \le n+1$ and $r \in \bbZ$.
First of all, for $r$ big enough, we have
$$
0=e_{1,j}(r)Y=\sum_{\alpha \in I}Y_{\alpha}A_j^{\alpha}(r)
+ \sum_{\alpha \in I}Y_{\alpha}B_j^{\alpha}(r)=\sum_{\alpha \in I}Y_{\alpha}B_j^{\alpha}(r) \, .
$$
From this relation, we can easily see
that, for all $r \in \bbZ$, $\sum\limits_{\alpha \in
I}Y_{\alpha}B_j^{\alpha}(r)=0$. Therefore
$$
\sum_{\alpha \in I}Y_{\alpha}A_j^{\alpha}(r)=0 \; ,
\forall \; r \in \bbZ, \, 2 \le j \le n+1 \, .
$$
Especially, $\sum\limits_{\alpha \in I}Y_{\alpha}A_j^{\alpha}
(-m^{\beta}_{j,i})=0$, for all $i,j \in \bbZ$ and $\beta \in I$, i.e.,
$$
\mu \sum_{\alpha \in I}Y_{\alpha} \# \{l \, | \,
m^{\alpha}_{j,l}=m^{\beta}_{j,i} \} e_{2,1}(m^{\alpha}_{2,k_{2,\alpha}})
\cdots \widehat{e_{j,1}(m^{\beta}_{j,i})} \cdots
e_{n+1,1}(m^{\alpha}_{n+1,k_{n+1,\alpha}})=0.
$$
Since all nonzero elements $e_{2,1}(m^{\alpha}_{2,k_{2,\alpha}}) \cdots
\widehat{e_{j,1}(m^{\beta}_{j,i})} \cdots e_{n+1,1}(m^{\alpha}_{n+1,k_{n+1,\alpha}})$
are linear independent and $k_{\alpha} \ge 1$,
we have that $\mu \ne 0$ implies $Y_{\alpha}=0$ for all
$\alpha \in I$. Thus $Y=Y_0 \in \bbC^*$,
i.e., any singular vector of $U$ is a nonzero scalar.
Therefore, if $\mu \neq 0$, then $\widehat{W}$ is an irreducible $\widehat{\mathfrak{gl}_{1|n}}(\bbC)$-module.

For the last one, let $\succ$ be the \textbf{lexicographical
total order} on $\bbZ^s$ for any $s\in\bbZ_+$, that is, for
$\mathbf{a},\,\mathbf{b}\in\bbZ^s$
$$
\mathbf{a} \succ\mathbf{b}\Leftrightarrow \text{ there exists } j \in\mathbb{N} \text{ such that } (a_i=b_i, \forall 1\leq i<j \leq s) \text{ and } a_j>b_j
$$
Any element in $\widehat{W}$ could be expressed in the following form
\begin{equation*}
\begin{split}
v=&\sum\limits_{\alpha\in I}A_{\alpha}x_2(m_{2,1}^{\alpha})\cdots x_2(m_{2,k_2}^{\alpha})\cdots x_{m}(m_{m,1}^{\alpha})\cdots x_m(m_{m,k_m}^{\alpha})\\
&x_{m+1}(m_{m+1,1}^{\alpha})\cdots x_{m+1}(m_{m+1,k_{m+1}}^{\alpha})\cdots x_{m+n}(m_{m+n,1}^{\alpha})\cdots x_{m+n}(m_{m+n,k_{m+n}}^{\alpha})\\
=&\sum\limits_{\alpha\in I'}x_{m+1}(m_{m+1,1}^{\alpha})\cdots x_{m+1}(m_{m+1,k_{m+1}}^{\alpha})Y_{\alpha}
\end{split}
\end{equation*}
where $I'$ is a subset of the finite index set $I$ and $Y_{\alpha}$ is a nonzero element without terms involving $x_{m+1}(*)$ for
any $\alpha\in I'$. Moreover, we may assume that
$$
m_{m+1,1}^{\alpha}>m_{m+1,2}^{\alpha}>\ldots>m_{m+1,k_{m+1}}^{\alpha}.
$$
Take a set of big enough integers
$b_{m+1,1}>b_{m+1,2}>\ldots>b_{m+1,k_{m+1}},$ then we have
\begin{equation*}
e_{2,m+1}(b_{m+1,k_{m+1}})\ldots e_{2,m+1}(b_{m+1,1})v =
x_2(m_{m+1,1}^{\beta}+b_{m+1,1})\cdots x_2(m_{m+1,k_{m+1}}^{\beta}+b_{m+1,k_{m+1}})Y_{\beta}+Y
\end{equation*}
where $\beta\in I'$ is the maximal element under the lexicographical total order of $\bbZ^{k_{m+1}}$.
Since, by the choice of $b_{m+1,i}$, we know $Y$ and
$x_2(m_{m+1,1}^{\beta}+b_{m+1,1})\cdots x_2(m_{m+1,k_{m+1}}^{\beta}+b_{m+1,k_{m+1}})Y_{\beta}$ are linearly independent,
hence $v$ could not be a singular vector of $\widehat{W}$ unless $k_{m+1}=0$.
Inductively, we can reduce the last case to the case of $\widehat{\mathfrak{gl}_{m|0}}(\bbC)$.
The proof is therefore completed.
\end{proof}

\end{document}